\documentclass[a4paper]{amsart}

\title{On cyclic fixed points of spectra}
\author[M.~B{\"o}kstedt]{Marcel B{\"o}kstedt}
\address{Department of Mathematical Sciences, Aarhus University,
        {\AA}rhus, Denmark}
\email{marcel@imf.au.dk}
\author[R.~R.~Bruner]{Robert R. Bruner}
\address{Department of Mathematics, Wayne State University, Detroit, USA}
\email{rrb@math.wayne.edu}
\author[S.~Lun{\o}e--Nielsen]{\\ Sverre Lun{\o}e--Nielsen}
\address{Department of Mathematics, University of Oslo, Norway}
\email{sverreln@math.uio.no}
\author[J.~Rognes]{John Rognes}
\address{Department of Mathematics, University of Oslo, Norway}
\email{rognes@math.uio.no}
\date{December 14th 2012}

\usepackage{amsmath,amssymb,mathrsfs,amsxtra,amsrefs}
\usepackage{fullpage}
\usepackage[all]{xy}
\newdir{ >}{{}*!/-5pt/@{>}} 
\SelectTips{cm}{} 
\renewcommand{\:}{\colon\thinspace} 

\newtheorem{theorem}{Theorem}[section]
\newtheorem{lemma}[theorem]{Lemma}
\newtheorem{proposition}[theorem]{Proposition}

\theoremstyle{definition}
\newtheorem{definition}[theorem]{Definition}
\newtheorem{example}[theorem]{Example}
\numberwithin{equation}{section}

\DeclareMathOperator*{\colim}{colim}
\DeclareMathOperator{\Ext}{Ext}

\DeclareMathOperator*{\hocolim}{hocolim}
\DeclareMathOperator{\hofib}{hofib}
\DeclareMathOperator*{\holim}{holim}
\DeclareMathOperator{\Hom}{Hom}
\DeclareMathOperator{\Map}{Map}
\newcommand{\C}{\mathbb{C}}
\newcommand{\F}{\mathbb{F}}
\newcommand{\T}{\mathbb{T}}
\newcommand{\Z}{\mathbb{Z}}
\newcommand{\A}{\mathscr{A}}
\newcommand{\longto}{\longrightarrow}
\newcommand{\tH}{\widehat{H}}

\begin{document}

\begin{abstract}
For a finite $p$-group $G$ and a bounded below $G$-spectrum $X$
of finite type mod~$p$, the $G$-equivariant Segal conjecture for $X$
asserts that the canonical map $X^G \to X^{hG}$, from $G$-fixed points to
$G$-homotopy fixed points, is a $p$-adic equivalence.  Let $C_{p^n}$ be
the cyclic group of order~$p^n$.  We show that if the $C_p$-equivariant
Segal conjecture holds for a $C_{p^n}$-spectrum $X$, as well as for
each of its geometric fixed point spectra $\Phi^{C_{p^e}}(X)$ for $0 <
e < n$, then the $C_{p^n}$-equivariant Segal conjecture holds for~$X$.
Similar results also hold for weaker forms of the Segal conjecture,
asking only that the canonical map induces an equivalence in sufficiently
high degrees, on homotopy groups with suitable finite coefficients.
\end{abstract}

\maketitle{}

\section{Introduction} \label{sec-1}

Let $p$ be any prime number.  Graeme Segal's Burnside ring conjecture
\cite{Ad82} for a finite $p$-group $G$ asserts that if $X = S_G$ is the
genuinely $G$-equivariant sphere spectrum, then the canonical map $X^G \to
X^{hG} = F(EG_+, X)^G$ is a $p$-adic equivalence.  For cyclic groups $G =
C_p$ of prime order the conjecture was proved by Lin \cite{LDMA80} and
Gunawardena \cite{Gu80}, \cite{AGM85}.  Thereafter Ravenel \cite{Ra81},
\cite{Ra84} gave an inductive proof of Segal's conjecture for finite
cyclic $p$-groups $G = C_{p^n}$ of order~$p^n$, starting from Lin and
Gunawardena's theorems.
Ravenel's result was superseded by Carlsson's proof \cite{Ca84} of the
Segal conjecture for all finite $p$-groups, but as we shall show here,
Ravenel's methods are also of interest in a more general context, where
$X$ is a quite general $G$-spectrum.

\medskip

As was elucidated by Miller and
Wilkerson \cite{MW83}, Ravenel's methods give two proofs of the Segal
conjecture for cyclic groups---one computational using the modified
Adams spectral sequence, and one non-computational, using explicit
geometric constructions.
The object of this paper is to generalize Ravenel's geometric proof of
the Segal conjecture to show that $X^G \to X^{hG}$ is ``close to'' a
$p$-adic equivalence for $G = C_{p^n}$, assuming that $X^C \to X^{hC}$ and
similar maps are ``close to'' such an equivalence for $C = C_p$.  Our main
technical results are Theorems~\ref{thm-1.6} and~\ref{cor-1.7}.
Their statements involve $(W, k)$-coconnected maps and geometric
fixed points, which are discussed in Definitions~\ref{def:coconn}
and~\ref{dfn-1.3}, respectively.  See Example~\ref{ex-2.2} for more on
how a $(W, k)$-coconnected map is close to a $p$-adic equivalence.

\medskip

In the special cases $X = B^{\wedge p^n}$ and $X = THH(B)$, where
$B^{\wedge p^n}$ is a specific $C_{p^n}$-equivariant model for the
$p^n$-th smash power of a symmetric spectrum $B$, and $THH(B)$ is
the topological Hochschild homology of a symmetric ring spectrum
$B$, the geometric fixed points are well understood, as explained
in Theorems~\ref{thm-1.9} and~\ref{thm-1.10}, respectively.  In the
special cases $W = S^{-1}/p^\infty$ and $W = F(V, S)$, where $V$ is
a finite $p$-torsion spectrum, the $(W, k)$-coconnected maps are well
understood in terms of $p$-completion and homotopy with $V$-coefficients,
as explained in Examples~\ref{ex-1.11} and~\ref{ex-1.12}, respectively.
In the doubly special case when $X = THH(B)$ and $W = S^{-1}/p^\infty$,
our results recover the main theorem of Tsalidis \cite{Ts98}.

\section{Statement of results}

We first formalize the notion of being close to a $p$-adic equivalence.
Throughout the paper we assume that a pair $(W, k)$ has been chosen as
in the following definition.  The hypothesis on~$W$ ensures that the
function spectrum $F(W, Y)$ is contractible whenever the
$p$-adic completion $Y\sphat_p$ is contractible.

\begin{definition} \label{def:coconn}
Let $S^{-1}\!/p^\infty$ be a Moore spectrum with homology $\Z/p^\infty$
concentrated in degree~$-1$, so that $F(S^{-1}\!/p^\infty, Y)
= Y\sphat_p$ for each spectrum $Y$.  Let $W$ be an object in the
localizing ideal \cite{HPS97}*{Def.~1.4.3(d)} of spectra generated by
$S^{-1}\!/p^\infty$, i.e., the smallest thick subcategory of spectra that
contains $S^{-1}\!/p^\infty$ and is closed under arbitrary wedge sums,
as well as under smash products with arbitrary spectra.

Let $k$ be an integer, or the symbol $-\infty$.  We say that a spectrum
$Y$ is \emph{$(W, k)$-coconnected} if $\pi_* F(W, Y) = 0$ for all $*
\ge k$.  We say that a map of spectra $f \: Y_1 \to Y_2$ is \emph{$(W,
k)$-coconnected} if $\hofib(f)$ is $(W, k)$-coconnected, or equivalently,
if $\pi_* F(W, Y_1) \to \pi_* F(W, Y_2)$ is injective for $* = k$ and
an isomorphism for all $* > k$.
\end{definition}

\begin{example} \label{ex-2.2}
The most obvious choice for $W$ is the Moore spectrum
$S^{-1}\!/p^\infty$ itself, in which
case $F(W, Y) = Y\sphat_p$, so a map $f \: Y_1 \to Y_2$ is $(W,
k)$-coconnected if and only if the $p$-completed map $f\sphat_p \:
(Y_1)\sphat_p \to (Y_2)\sphat_p$ induces an injection on $\pi_*$ for $*
= k$ and an isomorphism for $* > k$.  When $k=-\infty$, this is the
same as being a $p$-adic equivalence.

Alternatively, we may take $W = F(V, S)$, where $V$ is a finite CW
spectrum whose integral homology is $p$-torsion, in which case $F(W, Y)
\simeq V \wedge Y$ by Spanier--Whitehead duality.
In this case $f \:
Y_1 \to Y_2$ is $(W, k)$-coconnected if and only if the map $1 \wedge
f \: V \wedge Y_1 \to V \wedge Y_2$ induces an injection $V_*(Y_1) =
\pi_*(V \wedge Y_1) \to \pi_*(V \wedge Y_2) = V_*(Y_2)$ for $* = k$
and an isomorphism for $* > k$.
The Smith--Toda complexes $V(m)$ for $m\ge0$,
see \cite{Sm70} and \cite{To71},
are examples of such finite $p$-torsion spectra.
\end{example}

Next, we recall some comparison maps between fixed points, homotopy
fixed points, geometric fixed points and Tate constructions.

\begin{definition} \label{dfn-1.3}
Let $C = C_p \subset C_{p^n} = G$ and $\bar G = G/C \cong C_{p^{n-1}}$.
Let $\lambda = \C(1)$ be the basic faithful $G$-representation of
complex rank one, and $S^\lambda$ its one-point compactification.
Let $\infty\lambda$ be the direct sum of a countable number of
copies of~$\lambda$.  Its unit sphere $S(\infty\lambda) = EG$ is a
free contractible $G$-CW space, and its one-point compactification
$S^{\infty\lambda} = \widetilde{EG}$ sits in a $G$-homotopy cofiber
sequence $EG_+ \to S^0 \to \widetilde{EG}$, where the first map collapses
$EG$ to the non-basepoint.

Let $X$ be a $G$-spectrum, in the sense of \cite{LMS86}, and consider
the vertical map
$$
\xymatrix{
EG_+ \wedge X \ar[r] \ar[d]^{\simeq_G} & X \ar[r] \ar[d] &
	{}\widetilde{EG} \wedge X \ar[d] \\
EG_+ \wedge F(EG_+, X) \ar[r] & F(EG_+, X) \ar[r] &
	{}\widetilde{EG} \wedge F(EG_+, X)
}
$$
of horizontal $G$-homotopy cofiber sequences.  Passing to $G$-fixed
point spectra we obtain a vertical map
$$
\xymatrix{
X_{hG} \ar[rr]^-N \ar@{=}[d] && X^G \ar[rr]^-R \ar[d]^{\Gamma_n} &&
	\Phi^C(X)^{\bar G} \ar[d]^{\hat\Gamma_n} \\
X_{hG} \ar[rr]^-{N^h} && X^{hG} \ar[rr]^-{R^h} && X^{tG}
}
$$
of horizontal homotopy cofiber sequences, often called the
norm--restriction sequences \cite{GM95}*{Diag.~(C), (D)}.
Here
\begin{align*}
X_{hG} &= EG_+ \wedge_G X	&&\qquad\text{(homotopy orbits)} \\
X^{hG} &= F(EG_+, X)^G		&&\qquad\text{(homotopy fixed points)} \\
X^{tG} &= [\widetilde{EG} \wedge F(EG_+,X)]^G
				&&\qquad\text{(Tate construction)} \\
\intertext{and there is a preferred $\bar G$-equivariant equivalence}
[\widetilde{EG} \wedge X]^C &\overset{\simeq}\longto \Phi^C(X)
				&&\qquad\text{(geometric fixed points)}
\end{align*}
inducing the upper right hand equivalence $[\widetilde{EG}
\wedge X]^G \simeq \Phi^C(X)^{\bar G}$.  For more details, see
e.g.~\cite{HM97}*{Prop.~2.1}.
\end{definition}

The right hand square above is homotopy cartesian, so $\Gamma_n$ is $(W,
k)$-co\-connected if and only if $\hat\Gamma_n$ is $(W, k)$-coconnected.
This observation can be combined with the conclusions of all of the
theorems below.

We briefly write $H_*(X) = H_*(X; \F_p)$ for the mod~$p$
homology of any spectrum, and say that $H_*(X)$ is of finite type if
each group $H_m(X)$ is finite.

\begin{theorem} \label{thm-1.6}
Let $X$ be a $G$-spectrum with $\pi_*(X)$ bounded
below and $H_*(X)$ of finite type.  Suppose that
$\Gamma_1 \: X^C \to X^{hC}$
and
$\Gamma_{n-1} \: \Phi^C(X)^{\bar G} \to \Phi^C(X)^{h\bar G}$
are $(W, k)$-coconnected maps.  Then
$\Gamma_n \: X^G \to X^{hG}$
is $(W, k)$-coconnected.
\end{theorem}

Informally, this theorem asserts that if $X^C \to X^{hC}$ and $Y^{\bar G}
\to Y^{h\bar G}$ are close to $p$-adic equivalences, for $Y = \Phi^C(X)$,
then $X^G \to X^{hG}$ is close to a $p$-adic equivalence.

\begin{theorem} \label{cor-1.7}
Let $X$ be a $C_{p^n}$-spectrum.  Suppose, for each of the
geometric fixed point spectra
$$
Y = X \,,\, \Phi^{C_p}(X) \,,\, \dots \,,\, \Phi^{C_{p^{n-1}}}(X) \,,
$$
that $\pi_*(Y)$ is bounded below, $H_*(Y)$ is of finite type and $\Gamma_1
\: Y^{C_p} \to Y^{hC_p}$ is $(W, k)$-coconnected.  Then $\Gamma_n \:
X^{C_{p^n}} \to X^{hC_{p^n}}$
is $(W, k)$-coconnected.
\end{theorem}

The proofs of Theorems~\ref{thm-1.6} and~\ref{cor-1.7} are given near
the end of Section~\ref{sec-2}.  The bounded below and finite type
mod~$p$ hypotheses enter in the proof of Proposition~\ref{prop-2.10},
where we make use of the convergence of an inverse limit of Adams spectral
sequences.

The following construction was introduced by the first author, in the
context of functors with smash product (FSPs).
See \cite{BHM89}*{\S3} for a published account.
We are principally interested in the case $r = p^n$.

\begin{definition}
Let $B$ be any symmetric spectrum.  The $r$-th smash power $B^{\wedge
r}$ can be defined as a $C_r$-spectrum by the construction
$$
B^{\wedge r} = sd_r THH(B)_0
	= THH(B)_{r-1}
$$
from \cite{HM97}*{\S2.4}.  Its $V$-th space is defined by a homotopy
colimit
$$
(B^{\wedge r})_V = \hocolim_{(i_1, \dots, i_r) \in I^r}
\Map(S^{i_1} \wedge \dots \wedge S^{i_r},
B_{i_1} \wedge \dots \wedge B_{i_r} \wedge S^V) \,,
$$
and $C_r$ cyclically permutes the smash factors, in addition to its
natural action on $S^V$.
\end{definition}

To ensure that $B^{\wedge r}$ has the same naively equivariant
homotopy type as the ordinary $r$-fold smash product $B \wedge \dots
\wedge B$, it suffices to assume that $B$ is flat and convergent, see
e.g.~\cite{LR12}*{Lem.~5.5}.  Hereafter, when referring to $B^{\wedge r}$
we always assume that $B$ has first been replaced by an equivalent flat
and convergent symmetric spectrum.
In \cite{LR12}*{Thm.~5.13}, the third and fourth authors prove that
$\Gamma_1 \: (B^{\wedge p})^{C_p} \to (B^{\wedge p})^{hC_p}$ is a $p$-adic
equivalence whenever $\pi_*(B)$ is bounded below and $H_*(B)$ is of
finite type.  This provides the inductive beginning for the following
application of Theorem~\ref{cor-1.7}.

\begin{theorem} \label{thm-1.9}
Let $B$ be a symmetric spectrum with $\pi_*(B)$ bounded below and $H_*(B)$
of finite type.  Then
$$
\Gamma_n \: (B^{\wedge {p^n}})^{C_{p^n}}
	\to (B^{\wedge {p^n}})^{hC_{p^n}}
$$
is a $p$-adic equivalence, for each $n\ge1$.
\end{theorem}

When $B$ is a symmetric ring spectrum, its topological Hochschild
homology $THH(B)$ is a $\T$-spectrum \cite{HM97}*{\S2.4}, where $\T$
is the circle group.  It is not true in general that $\Gamma_1 \:
THH(B)^{C_p} \to THH(B)^{hC_p}$ is a $p$-adic equivalence, see
e.g.~\cite{HM97}*{Prop.~5.3} and \cite{Ro99}*{Thm.~4.7}, but when it
is ``approximately'' true, then the following theorem is useful.

\begin{theorem} \label{thm-1.10}
Let $B$ be a connective symmetric ring spectrum with $H_*(B)$ of finite
type, and suppose that
$$
\Gamma_1 \: THH(B)^{C_p} \to THH(B)^{hC_p}
$$
is $(W, k)$-coconnected.  Then
$$
\Gamma_n \: THH(B)^{C_{p^n}} \to THH(B)^{hC_{p^n}}
$$
is $(W, k)$-coconnected, for each $n\ge2$.
\end{theorem}

The proofs of Theorems~\ref{thm-1.9} and~\ref{thm-1.10} are given at
the end of Section~\ref{sec-2}.

In the case $B = S$ there is a $G$-equivariant equivalence $THH(S) \simeq
S_G$, and $\Gamma_1$ is a $p$-adic equivalence by the classical Segal
conjecture.  Also in the cases $B = MU$ (the complex cobordism spectrum)
and $B = BP$ (the Brown--Peterson spectrum) it turns out that $\Gamma_1$
for $THH(B)$ is a $p$-adic equivalence, as the third and fourth authors
show in \cite{LR11}*{Thm.~1.1}.  This provides examples with
$k=-\infty$ for the following special case.

\begin{example} \label{ex-1.11}
Taking $W = S^{-1}\!/p^\infty$, the assumption in Theorem~\ref{thm-1.10} is
that the $p$-completed map $\Gamma_1 \: (THH(B)^{C_p})\sphat_p \to
(THH(B)^{hC_p})\sphat_p$ is $k$-coconnected, i.e., that it induces
an injection on $\pi_k$ and an isomorphism on $\pi_*$ for $* >
k$, and the conclusion is that the $p$-completed map
$$
\Gamma_n \: (THH(B)^{C_{p^n}})\sphat_p \to (THH(B)^{hC_{p^n}})\sphat_p
$$
is also $k$-coconnected, for all $n\ge2$.  This recovers a theorem of
Tsalidis \cite{Ts98}*{Thm.~2.4}.
\end{example}

\begin{example} \label{ex-1.12}
Taking $W = F(V, S)$ and $V = V(1) = S/(p, v_1)$, the Smith--Toda
complex of chromatic type~$2$ (for $p$ odd), the assumption in
Theorem~\ref{thm-1.10} is that
$$
V(1)_*(\Gamma_1) \: V(1)_* THH(B)^{C_p} \to V(1)_* THH(B)^{hC_p}
$$
is $k$-coconnected, and the conclusion is that
$$
V(1)_*(\Gamma_n) \:  V(1)_* THH(B)^{C_{p^n}} \to V(1)_* THH(B)^{hC_{p^n}}
$$
is
also $k$-coconnected, for all $n\ge2$.  This recovers the generalization
of Tsalidis' theorem used by Ausoni and the fourth author
\cite{AR02}*{Thm.~5.7} in the special case when $B = \ell$, the Adams
summand of connective $p$-local complex $K$-theory, and $k = 2p-2$.
The generalized result is used again in \cite{AR12}*{Cor.~5.9},
for $B = \ell/p = k(1)$, the first connective Morava $K$-theory.
\end{example}

\section{Constructions and proofs} \label{sec-2}

\begin{definition}
Let $\bar\lambda$ be the basic faithful $\bar G$-representation of
complex rank one.  Like in Definition~\ref{dfn-1.3}, we let $E\bar G =
S(\infty\bar\lambda)$ and $\widetilde{E\bar G} = S^{\infty\bar\lambda}$.
The usual map from the homotopy colimit to the categorical colimit
is a $\bar G$-equivalence $\hocolim_j S^{j\bar\lambda}
\overset{\simeq}\longto S^{\infty\bar\lambda} = \widetilde{E\bar G}$.
The pullback of $\bar\lambda$ along the canonical projection $G \to
\bar G$ is the $p$-th tensor power $\lambda^p = \C(p)$ of $\lambda$,
and we get a $G$-equivalence
$$
\hocolim_j S^{j\lambda^p} \overset{\simeq}\longto S^{\infty\lambda^p}
	= \widetilde{E\bar G} \,,
$$
where the right hand side is implicitly viewed as a $G$-space by pullback
along $G \to \bar G$.  The $G$-map $S^{j\lambda^p} \to S^{(j+1)\lambda^p}$
in the colimit system is given by smashing $S^{j\lambda^p}$ with the
one-point compactification $z \: S^0 \to S^{\lambda^p}$ of the inclusion
$\{0\} \subset \lambda^p$.
\end{definition}

\begin{lemma} \label{lem-2.2}
Let $X$ be a $G$-spectrum.  There is a natural homotopy cofiber
sequence
$$
\xymatrix{
\holim_j \, (\Sigma^{-j\lambda^p} X)^G \ar[r]
& (X^C)^{\bar G} \ar[rr]^-{\Gamma_{n-1}}
&& (X^C)^{h\bar G} \,,
}
$$
where the right hand map is $\Gamma_{n-1}$ for the $\bar G$-spectrum
$X^C$.
\end{lemma}

\begin{proof}
By mapping the $\bar G$-homotopy cofiber sequence
$E\bar G_+ \to S^0 \to \widetilde{E\bar G}$ into $X^C$, we get the
homotopy (co-)fiber sequence
$$
F(\widetilde{E\bar G}, X^C)^{\bar G} \to (X^C)^{\bar G}
	\overset{\Gamma_{n-1}}\longto F(E\bar G_+, X^C)^{\bar G} \,.
$$
At the left hand side we have a natural chain of equivalences
\begin{multline*}
F(\widetilde{E\bar G}, X^C)^{\bar G}
\simeq
F(\hocolim_j S^{j\bar\lambda}, X^C)^{\bar G}
\simeq
\holim_j F(S^{j\bar\lambda}, X^C)^{\bar G} \\
\simeq
\holim_j (\Sigma^{-j\bar\lambda} (X^C))^{\bar G}
\simeq
\holim_j ((\Sigma^{-j\lambda^p} X)^C)^{\bar G}
\simeq
\holim_j \, (\Sigma^{-j\lambda^p} X)^G \,.
\end{multline*}
This gives the asserted homotopy cofiber sequence.
\end{proof}

\begin{proposition} \label{prop-2.3}
Let $X$ be a $G$-spectrum.  There is a vertical map
of homotopy cofiber sequences
$$
\xymatrix{
\holim_j \Phi^C(\Sigma^{-j\lambda^p} X)^{\bar G} \ar[r] \ar[d]
& \Phi^C(X)^{\bar G} \ar[rr]^-{\Gamma_{n-1}} \ar[d]^{\hat\Gamma_n}
&& \Phi^C(X)^{h\bar G} \ar[d]^{(\hat\Gamma_1)^{h\bar G}} \\
\holim_j \, (\Sigma^{-j\lambda^p} X)^{tG} \ar[r]
& X^{tG} \ar[rr]^-{\Gamma_{n-1}}
&& (X^{tC})^{h\bar G} \rlap{\,.}
}
$$
The right hand horizontal maps are $\Gamma_{n-1}$ for the $\bar G$-spectra
$\Phi^C(X) \simeq [\widetilde{EG} \wedge X]^C$ and
$X^{tC} = [\widetilde{EG} \wedge F(EG_+, X)]^C$, respectively.
\end{proposition}

\begin{proof}
We replace $X$ in the lemma above by the $G$-spectra $\widetilde{EG}
\wedge X$ and $\widetilde{EG} \wedge F(EG_+, X)$.  This gives the two
claimed homotopy cofiber sequences, in view of the $\bar G$-equivalences
$$
[\Sigma^{-j\lambda^p} (\widetilde{EG} \wedge X)]^C
\simeq [\widetilde{EG} \wedge \Sigma^{-j\lambda^p} X]^C
\simeq \Phi^C(\Sigma^{-j\lambda^p} X)
$$
and
$$
[\Sigma^{-j\lambda^p} (\widetilde{EG} \wedge F(EG_+, X))]^C
\simeq [\widetilde{EG} \wedge F(EG_+, \Sigma^{-j\lambda^p} X)]^C
= (\Sigma^{-j\lambda^p} X)^{tC} \,,
$$
respectively.  These follow from the $G$-dualizability of
$S^{j\lambda^p}$.
\end{proof}

\begin{lemma} \label{lem-2.4}
If $\hat\Gamma_1 \: \Phi^C(X) \to X^{tC}$ is $(W, k)$-coconnected,
then $(\hat\Gamma_1)^{h\bar G}$ is $(W, k)$-coconnected.
\end{lemma}

\begin{proof}
This is a special case of a more general result.  The homotopy
fixed point spectral sequence
$$
E^2_{s,t} = H^{-s}(G; \pi_t(Y)) \Longrightarrow \pi_{s+t}(Y^{hG})
$$
shows that $Y^{hG}$ is $k$-coconnected whenever $Y$ is a $k$-coconnected
$G$-spectrum.  Commutation of function spectra, homotopy fibers and
homotopy fixed points shows that $f^{hG} \: Y_1^{hG} \to Y_2^{hG}$ is $(W,
k)$-coconnected whenever $f \: Y_1 \to Y_2$ is a $(W, k)$-coconnected
$G$-map.  The lemma follows by applying this to the case of the $\bar
G$-map~$\hat\Gamma_1$.
\end{proof}

\begin{definition}
The \emph{Greenlees filtration} \cite{Gr87}*{p.~437} of $\widetilde{EG}
= S^{\infty\lambda}$ is an integer-indexed $G$-cellular filtration
of spectra, whose $2i$-th term is $S^{i\lambda}$ for each $i$.
The $(2i+1)$-th term is obtained from $S^{i\lambda}$ by attaching a single
$G$-free $(2i+1)$-cell, and $S^{(i+1)\lambda}$ is in turn obtained from
it by attaching a single $G$-free $(2i+2)$-cell.  The composite $G$-map
$S^{i\lambda} \to S^{(i+1)\lambda}$ is given by smashing $S^{i\lambda}$
with the one-point compactification $\tau \: S^0 \to S^\lambda$ of the
inclusion $\{0\} \subset \lambda$.  The Greenlees filtration induces
an increasing filtration of $X^{tG} = [\widetilde{EG} \wedge F(EG_+,
X)]^G$, and a tower of homotopy cofibers with $(2i+1)$-th term
\begin{equation}
	\label{eq-2.6}
X^{tG}\langle i\rangle = [\widetilde{EG}/S^{i\lambda} \wedge F(EG_+, X)]^G
\,,
\end{equation}
which we call the Tate tower.  The associated spectral
sequence is the homological $G$-equivariant Tate spectral sequence
$$
\hat E^2_{s,t} = \tH^{-s}(G; H_t(X))
$$
converging to the continuous homology groups
$$
H^c_*(X^{tG}) = \lim_i H_*(X^{tG}\langle i\rangle)
$$
of $X^{tG}$, when $X$ is a bounded below spectrum with $H_*(X)$
of finite type.  See \cite{LR12}*{Def.~2.3, Prop.~4.15}.  Note that $i$
tends to $-\infty$ in this limit.  We shall also refer to the continuous
cohomology groups
$$
H_c^*(X^{tG}) = \colim_i H^*(X^{tG}\langle i\rangle) \,,
$$
and note that $H^c_*(X^{tG}) \cong H_c^*(X^{tG})^*$ (the $\Hom$ dual)
when $H_*(X)$ is bounded below and of finite type, because then
each $H_*(X^{tG}\langle i\rangle)$ is also of finite type.
\end{definition}

\begin{definition}
Let the $G$-map $\xi \: S^\lambda \to S^{\lambda^p}$ of representation
spheres be the suspension of the standard degree~$p$ covering map $\pi
\: S(\lambda) \to S(\lambda^p)$ of unit circles, as in the following
vertical map of horizontal $G$-homotopy cofiber sequences:
$$
\xymatrix{
S(\lambda)_+ \ar[r] \ar[d]^{\pi_+} & S^0 \ar[r]^-{\tau} \ar@{=}[d]
	& S^\lambda \ar[d]^\xi \\
S(\lambda^p)_+ \ar[r] & S^0 \ar[r]^-z & S^{\lambda^p} \,.
}
$$
\end{definition}

We note that $\xi$ is not induced by a linear map.  It has degree~$p$
on the top cell, so $\xi_* \: H_*(S^\lambda) \to H_*(S^{\lambda^p})$ is
the zero homomorphism, since we work with reduced homology and mod~$p$
coefficients.

\begin{proposition} \label{prop-2.8}
Let $X$ be a $G$-spectrum with $H_*(X)$ bounded below.  Then
$$
\lim_j H^c_*((\Sigma^{-j\lambda^p} X)^{tG}) = \lim_{i,j}
H_*((\Sigma^{-j\lambda^p} X)^{tG}\langle i\rangle) = 0
$$
and
$$
\colim_j H_c^*((\Sigma^{-j\lambda^p} X)^{tG}) = \colim_{i,j}
H^*((\Sigma^{-j\lambda^p} X)^{tG}\langle i\rangle) = 0 \,.
$$
\end{proposition}

\begin{proof}
In the notation of~\eqref{eq-2.6} we have a natural equivalence
$$
\nu \: (\Sigma^{-j\lambda^p} X)^{tG}\langle i\rangle
	\overset{\simeq}\longto
	(\Sigma^{j(\lambda-\lambda^p)} X)^{tG}\langle i{-}j\rangle
$$
for each $i$ and $j$.
It is obtained from the $G$-equivalence
$$
\frac{S^{\infty\lambda} \wedge S^{j\lambda}}{S^{i\lambda}}
\wedge F(EG_+, \Sigma^{-j\lambda^p}X)
\overset{\simeq}\longto
\frac{S^{\infty\lambda}}{S^{(i-j)\lambda}}
\wedge F(EG_+, S^{j\lambda} \wedge \Sigma^{-j\lambda^p}X)
$$
(see \cite{LMS86}*{III.1}) by passage to $G$-fixed points.
Under these equivalences, the $z$-tower map
$$
z \: (\Sigma^{-(j+1)\lambda^p} X)^{tG}\langle i\rangle
	\to (\Sigma^{-j\lambda^p} X)^{tG}\langle i\rangle
$$
induced by smashing with $z \: S^0 \to S^{\lambda^p}$ is
compatible with the composite of the Tate tower map
$$
\tau \: (\Sigma^{(j+1)(\lambda-\lambda^p)} X)^{tG}\langle i{-}j{-}1\rangle
	\to (\Sigma^{(j+1)(\lambda-\lambda^p)} X)^{tG}\langle i{-}j\rangle
$$
induced by smashing with $\tau \: S^0 \to S^\lambda$, and the
$\xi$-tower map
$$
\xi \: (\Sigma^{(j+1)(\lambda-\lambda^p)} X)^{tG}\langle i{-}j\rangle
	\to (\Sigma^{j(\lambda-\lambda^p)} X)^{tG}\langle i{-}j\rangle
$$
induced by smashing with $\xi \: S^\lambda \to S^{\lambda^p}$, in
the sense that the following diagram commutes up to homotopy:
$$
\xymatrix@C-3.8pc{
[\widetilde{EG}/S^{i\lambda} \wedge F(EG_+, \Sigma^{-(j+1)\lambda^p} X)]^G
\ar[dr]^-{\nu}_-{\simeq} \ar[ddd]_{z} \\
& [\widetilde{EG}/S^{(i-j-1)\lambda} \wedge
	F(EG_+, \Sigma^{(j+1)(\lambda-\lambda^p)} X)]^G \ar[d]^{\tau} \\
& [\widetilde{EG}/S^{(i-j)\lambda} \wedge
	F(EG_+, \Sigma^{(j+1)(\lambda-\lambda^p)} X)]^G \ar[dd]^{\xi} \\
[\widetilde{EG}/S^{i\lambda} \wedge F(EG_+, \Sigma^{-j\lambda^p} X)]^G
\ar[dr]_-{\nu}^-{\simeq} \\
& [\widetilde{EG}/S^{(i-j)\lambda} \wedge
	F(EG_+, \Sigma^{j(\lambda-\lambda^p)} X)]^G \,.
}
$$
To see this, note that the diagram
$$
\xymatrix{
S^{\lambda} \wedge F(EG_+, S^0 \wedge X)
\ar[r]_-{\simeq}^-{\nu} \ar[dd]_{z}
& S^0 \wedge
	F(EG_+, S^\lambda \wedge X) \ar[d]^{\tau} \\
& S^\lambda \wedge
	F(EG_+, S^\lambda \wedge X) \ar[d]^{\xi} \\
S^\lambda \wedge F(EG_+, S^{\lambda^p} \wedge X)
\ar@{=}[r]
& S^\lambda \wedge
	F(EG_+, S^{\lambda^p} \wedge X)
}
$$
commutes up to $G$-homotopy (because the twist map $S^\lambda \wedge
S^\lambda \cong S^\lambda \wedge S^\lambda$ is $G$-homotopic to the
identity), replace $X$ with $\Sigma^{j\lambda - (j+1)\lambda^p}X$, smash
with $\widetilde{EG}/S^{(i-j-1)\lambda}$ and pass to $G$-fixed points.

Passing to homology, we get that $z_*$ is strictly compatible with
the composite $\xi_* \tau_*$ under the isomorphisms $\nu_*$.
Next pass to continuous homology by forming limits over $i$ along the
homomorphisms $\tau_*$.  Then the homomorphism
$$
z_* \: H^c_*((\Sigma^{-(j+1)\lambda^p} X)^{tG})
	\to H^c_*((\Sigma^{-j\lambda^p} X)^{tG})
$$
is identified with the homomorphism
\begin{equation}
	\label{eq-2.9}
\xi_* \: H^c_*((\Sigma^{(j+1)(\lambda-\lambda^p)} X)^{tG})
	\to H^c_*((\Sigma^{j(\lambda-\lambda^p)} X)^{tG}) \,,
\end{equation}
so it suffices to show that the limit over $j$ of the latter homomorphisms
is zero.  Let
$$
\hat E^2_{s,t}(j) = \tH^{-s}(G; H_t(\Sigma^{j(\lambda-\lambda^p)} X))
\Longrightarrow H^c_{s+t}((\Sigma^{j(\lambda-\lambda^p)} X)^{tG})
$$
be the homological Tate spectral sequence for the $j$-th term in the
$\xi$-tower.

By naturality of the Tate spectral sequence,
the homomorphism $\xi_*$ above is compatible with the spectral sequence
map $\hat E^2_{**}(j+1) \to \hat E^2_{**}(j)$ that is induced on
Tate cohomology by the $G$-module homomorphism
$$
\xi_* \: H_*(\Sigma^{(j+1)(\lambda-\lambda^p)} X)
	\to H_*(\Sigma^{j(\lambda-\lambda^p)} X) \,.
$$
This homomorphism is zero, since $\xi_* \: H_*(S^\lambda) \to
H_*(S^{\lambda^p})$ is zero.  Hence the map of spectral sequences is also
zero.  It follows that the homomorphism $\xi_*$ in~\eqref{eq-2.9} strictly
reduces the Tate filtration ($=s$) of each nonzero continuous homology
class. Equivalently, $\xi_*$ strictly increases the vertical degree
($=t$) of the spectral sequence representative of each nonzero class.

By assumption, there is an integer $\ell$ such that $H_t(X) = 0$ for
all $t < \ell$.  Then $\hat E^2_{s,t}(j) = \hat E^\infty_{s,t}(j) = 0$
for $t < \ell$ and any $j$.  If $x = (x_j)_j$ is an arbitrary element
of $\lim_j H^c_*((\Sigma^{j(\lambda-\lambda^p)} X)^{tG})$, then $x_j =
\xi_*^m(x_{j+m})$ for each $m\ge0$.  If $x_j$ is represented in vertical
degree~$t$, then $x_{j+m}$ must be represented in vertical degree $\le
(t-m)$.  Choosing $m$ so large that $t-m < \ell$, it follows that $x_{j+m}
= 0$, which implies $x_j = 0$.  Repeating the argument for each $j$
we see that $x=0$, so $\lim_j H^c_*((\Sigma^{j(\lambda-\lambda^p)}
X)^{tG})$ must be the trivial group.

Let $M = \colim_j H_c^*((\Sigma^{-j\lambda^p} X)^{tG})$.  Then the $\Hom$
dual $M^*$ is the limit group we just showed is zero, and $M$ injects
into its double $\Hom$ dual $M^{**}$, so $M = 0$ as well.
\end{proof}

\begin{proposition} \label{prop-2.10}
Let $X$ be a $G$-spectrum such that $\pi_*(X)$ is bounded below and $H_*(X)$
is of finite type.  Then the $p$-adic completion $Y\sphat_p$ of
$$
Y = \holim_j \, (\Sigma^{-j\lambda^p} X)^{tG}
$$
is contractible.  Hence the map $\Gamma_{n-1} \: X^{tG} \to
(X^{tC})^{h\bar G}$ is a $p$-adic equivalence.
\end{proposition}

\begin{proof}
The spectrum $Y$
is the homotopy limit over $i$ and $j$ of the spectra
$$
(\Sigma^{-j\lambda^p} X)^{tG}\langle i\rangle =
[\widetilde{EG}/S^{i\lambda} \wedge F(EG_+, \Sigma^{-j\lambda^p} X)]^G
\,,
$$
which can be rewritten as
$$
(\widetilde{EG}/S^{i\lambda} \wedge \Sigma^{-j\lambda^p} X)_{hG}
$$
by the Adams equivalence \cite{LMS86}*{II.8.4}, since
$\widetilde{EG}/S^{i\lambda}$ is a free $G$-CW spectrum.  Each of these
is bounded below with mod~$p$ homology of finite type.  Hence there is
an inverse limit Adams spectral sequence
$$
E_2^{**} = \Ext_{\A}^{**}(M, \F_p) \Longrightarrow \pi_*(Y\sphat_p)
$$
converging to the $p$-adic homotopy of that homotopy limit
(see \cite{CMP87}*{Prop.~7.1} and \cite{LR12}*{Prop.~2.2}), where
$$
M = \colim_j H_c^*((\Sigma^{-j\lambda^p} X)^{tG}) \,.
$$
The latter $\A$-module was shown to be zero in Proposition~\ref{prop-2.8},
hence the $E_2$-term is zero and $Y\sphat_p$ is contractible.
The second conclusion follows from Proposition~\ref{prop-2.3}.
\end{proof}

\begin{proof}[Proof of Theorem~\ref{thm-1.6}]
Consider the diagram in Proposition~\ref{prop-2.3}.  By assumption,
the maps
$$
\Gamma_{n-1} \: \Phi^C(X)^{\bar G} \to \Phi^C(X)^{h\bar G}
\qquad\text{and}\qquad
\Gamma_1 \: X^C \to X^{hC}
$$
are $(W, k)$-coconnected.  Hence $\hat\Gamma_1 \: \Phi^C(X) \to
X^{tC}$ is $(W, k)$-coconnected, so by Lemma~\ref{lem-2.4} also
$$
(\hat\Gamma_1)^{h\bar G} \: \Phi^C(X)^{h\bar G} \to (X^{tC})^{h\bar G}
$$
is $(W, k)$-coconnected.  By Proposition~\ref{prop-2.10}, the map
$\Gamma_{n-1} \: X^{tG} \to (X^{tC})^{h\bar G}$ is a $p$-adic equivalence,
hence $(W, -\infty)$-coconnected, by our standing assumption that $W$
is in the localizing ideal of spectra generated by $S^{-1}\!/p^\infty$.
It follows easily that $\hat\Gamma_n \: \Phi^C(X)^{\bar G} \to X^{tG}$
is $(W, k)$-coconnected, which is equivalent to $\Gamma_n \: X^G \to
X^{hG}$ being $(W, k)$-coconnected.
\end{proof}

\begin{proof}[Proof of Theorem~\ref{cor-1.7}]
This follows by induction on $n$, using Theorem~\ref{thm-1.6} and the
observation that
$$
\Phi^{C_p}(\Phi^{C_{p^e}}(X)) \cong \Phi^{C_{p^{e+1}}}(X)
$$
for all $0 \le e < n$.
\end{proof}


\begin{proof}[Proof of Theorem~\ref{thm-1.9}]
This follows from Theorem~\ref{cor-1.7} in the case $X = B^{\wedge
p^n}$, $W = S^{-1}\!/p^\infty$ and $k=-\infty$, once we show that for each
$0 \le e < n$ there is a $C_{p^{n-e}}$-equivalence
$$
Y = \Phi^{C_{p^e}}(B^{\wedge p^n}) \simeq B^{\wedge p^{n-e}} \,,
$$
the right hand side is bounded below with mod~$p$ homology of finite type,
and $\Gamma_1 \: Y^{C_p} \to Y^{hC_p}$ is a $p$-adic equivalence.
The first claim follows from the proof in simplicial degree~$0$
of \cite{HM97}*{Prop.~2.5}.  Writing $Y \simeq Z^{\wedge p}$, where
$Z = B^{\wedge p^{n-e-1}}$ is bounded below with $H_*(Z)$ of finite
type, the other claims also follow, since $\Gamma_1 \: (Z^{\wedge
p})^{C_p} \to (Z^{\wedge p})^{hC_p}$ is a $p$-adic equivalence by
\cite{LR12}*{Thm.~5.13}, generalizing \cite{BMMS86}*{\S II.5}.
\end{proof}

\begin{proof}[Proof of Theorem~\ref{thm-1.10}]
There is a $C_{p^{n-1}}$-equivalence
$$
r \: \Phi^{C_p} THH(B) \overset{\simeq}\longto THH(B)
$$
(the cyclotomic structure map of $THH(B)$, see \cite{HM97}*{\S2.5}), whose
$e$-fold iterate is a $C_{p^{n-e}}$-equivalence $\Phi^{C_{p^e}}(THH(B))
\simeq THH(B)$.  It is clear from the simplicial definition that $THH(B)$
is connective and has mod~$p$ homology of finite type, hence the theorem
follows from Theorem~\ref{cor-1.7}.
\end{proof}

\end{document}